\documentclass[11pt]{article}

\usepackage{amsmath, amsfonts,amssymb, amsthm}
\usepackage{graphicx}
\usepackage{color}
\usepackage{epsf}
\usepackage{tikz}
\usepackage{subfig}
\usepackage[left=1.3in, right=1.3in, top=1.25in, bottom=1.35in]{geometry}
\usepackage{hyperref}

\usepackage{enumitem}
\setlist{itemsep=0pt, topsep=0pt}

\newcommand{\floor}[1]{\lfloor#1\rfloor}

\newcommand{\tbf}[1]{\textbf{#1}}

\newtheorem{theorem}{Theorem}[section]
\newtheorem{corollary}[theorem]{Corollary}
\newtheorem{lemma}[theorem]{Lemma}

\newtheorem{proposition}[theorem]{Proposition}

\newtheorem{problem}[theorem]{Problem}

\newtheorem{conjecture}[theorem]{Conjecture}

\newcommand{\mc}{\mathrm{mc}}

\title{Large monochromatic components in 3-edge-colored Steiner triple systems}
\author{Louis DeBiasio$^{1}$, Michael Tait$^{2}$}
\date{\today}

\begin{document}

\maketitle
\noindent\footnotetext[1]{Department of Mathematics, Miami University {\tt debiasld@miamioh.edu}. Research supported in part by Simons Foundation Collaboration Grant \# 283194.}
\noindent\footnotetext[2]{Department of Mathematics \& Statistics, Villanova University {\tt michael.tait@villanova.edu}. Research supported in part by National Science Foundation Grant DMS-2011553.}

\begin{abstract}
It is known that in any $r$-coloring of the edges of a complete $r$-uniform hypergraph, there exists a spanning monochromatic component.  Given a Steiner triple system on $n$ vertices, what is the largest monochromatic component one can guarantee in an arbitrary 3-coloring of the edges?  

Gy\'arf\'as \cite{G2} proved that $(2n+3)/3$ is an absolute lower bound and that this lower bound is best possible for infinitely many $n$.  On the other hand, we prove that for almost all Steiner triple systems the lower bound is actually $(1-o(1))n$.  We obtain this result as a consequence of a more general theorem which shows that the lower bound depends on the size of a largest \emph{3-partite hole} (that is, disjoint sets $X_1, X_2, X_3$ with $|X_1|=|X_2|=|X_3|$ such that no edge intersects all of $X_1, X_2, X_3$) in the Steiner triple system (Gy\'arf\'as previously observed that the upper bound depends on this parameter).  Furthermore, we show that this lower bound is tight unless the structure of the Steiner triple system and the coloring of its edges are restricted in a certain way.

We also suggest a variety of other Ramsey problems in the setting of Steiner triple systems.
\end{abstract}

\section{Introduction}

Given a hypergraph $H$, let $H^{(2)}$ be the \emph{shadow-graph} of $H$; that is, the 2-uniform hypergraph on $V(H)$ where $\{x,y\}\in E(H^{(2)})$ if and only if there exists $e\in E(H)$ such that $\{x,y\}\subseteq e$.  We say that $H$ is \emph{connected} if $H^{(2)}$ is connected.  The \emph{components} of $H$ are the maximal connected subgraphs of $H$.
Given an $r$-coloring of $H$ (which for the purposes of this paper means a partition $\{E_1, \dots, E_r\}$ of $E(H)$ where we allow for some of the sets in the partition to be be empty) a \emph{monochromatic component} of $H$ is a component of the hypergraph $H_i=(V(H), E_i)$ for some $i\in [r]$.

For a hypergraph $H$, let $\mc_r(H)$ be the largest integer $m$ such that in every $r$-coloring of the edges of $H$, there exists a monochromatic component on at least $m$ vertices.  It is well known \cite{G} that for all $r\geq 2$, $\mc_r(K_n^r)=n$ where $K_n^r$ is the complete $r$-uniform hypergraph on $n$ vertices (when $r=2$ this is equivalent to the statement ``a graph or its complement is connected'').  

First note that for any hypergraph $H$, we have $\mc_r(H)\geq \mc_r(H^{(2)})$.  So if $H$ has the property that every pair of vertices is contained in at least one edge, then $\mc_r(H)\geq \mc_r(K_n)\geq \frac{n}{r-1}$, where the last inequality holds by a well-known result of Gy\'arf\'as \cite{G}.  

A Steiner triple system is a $3$-uniform hypergraph in which each pair of vertices is contained in exactly one edge, and it is well-known that a Steiner triple system on $n$ vertices exists if and only if $n\equiv 1,3\bmod 6$.  Given a Steiner triple system $S$, we sometimes write $|S|$ to denote the number of vertices in $S$ and we note that the number of edges in $S$ is exactly $\frac{1}{3}\binom{|S|}{2}=\frac{|S|(|S|-1)}{6}$.
For $n\equiv 1,3\bmod 6$, let $\mathcal{S}_n$ be the family of all Steiner triple systems on $n$ vertices.  The purpose of this paper is to address the following question:  Given $S\in \mathcal{S}_n$, what is the value of $\mc_3(S)$?  Note that by the observation in the previous paragraph, we trivially have $\mc_3(S)\geq \frac{n}{2}$ for all $S\in \mathcal{S}_n$.

This question unexpectedly led to the consideration of the following independence-type parameter. Given a hypergraph $H$ on $n$ vertices, let $\alpha_k^*(H)$ be the maximum integer $a_k^*$ such that there exist disjoint sets $X_1, X_2, \dots, X_k\subseteq V(H)$ with $|X_1|= |X_2|= \dots= |X_k|= a_k^*$ such that for all $e\in E(H)$, $e\cap X_i=\emptyset$ for some $i\in [k]$.  We call the sets $X_1, \dots, X_k$ a \emph{$k$-partite hole} and we say $\alpha^*_k$ is the \emph{$k$-partite-hole number} of $H$.  Note that since the sets are disjoint, we trivially have $\alpha^*_k(H)\leq \frac{n}{k}$.

We begin by noting the following simple upper bound on $\mc_k(H)$ for all $k$-uniform hypergraphs $H$ (essentially, this example was observed earlier in a more restricted form for the case of Steiner triple systems by Gy\'arf\'as \cite[Lemma 3]{G2}).

\begin{proposition}\label{n-a*_upper}
For all $k$-uniform hypergraphs $H$, $\mc_k(H)\leq n-\alpha_k^*(H)$.
\end{proposition}

\begin{proof}
Let $a_k:=\alpha_k^*(H)$.  Let $X_1, X_2, \dots, X_k\subseteq V(H)$ with $|X_1|= |X_2|= \dots =|X_k|= a_k$ such that for all $e\in E(H)$, $e\cap X_i=\emptyset$ for some $i\in [r]$.  Color every edge which avoids $X_i$ with color $i$ for all $i\in [k]$.  Note that every component of color $i$ avoids $X_i$, so $\mc_k(H)\leq n-a_k$.  
\end{proof}

After a draft of this paper appeared on arXiv, we were made aware of the fact that this problem and the above connection between the two parameters had been studied earlier by Gy\'arf\'as \cite{G2} (what we call a 3-partite hole, he calls a ``cross'').  He proved the following:

\begin{theorem}[Gy\'arf\'as {\cite[Theorem 1, Lemma 3, Theorem 4]{G2}}]\label{gy_thm}~
\begin{enumerate}
\item For $n\equiv 1,3\bmod 6$,  $\displaystyle \min_{S\in \mathcal{S}_n}\mc_3(S)\geq 2n/3+1$.
\item For $n\equiv 1,3\bmod 6$,  $\displaystyle \min_{S\in \mathcal{S}_n}\mc_3(S)\leq n-\max_{S\in \mathcal{S}_n} \alpha^*_3(S)$.
\item For $n\equiv 3 \bmod 18$, $\displaystyle\max_{S\in \mathcal{S}_n} \alpha^*_3(S)=n/3-1$; consequently, $\displaystyle \min_{S\in \mathcal{S}_n}\mc_3(S)=2n/3+1$.
\end{enumerate}
\end{theorem}

Our work can now be viewed as an extension of the results in \cite{G2}.  In Theorem \ref{gy_thm}(i), Gy\'arf\'as proves an absolute lower bound on $\mc_3(S)$.  Our first main result is a lower bound on $\mc_3(S)$ in terms of $\alpha^*_3(S)$.

\begin{theorem}\label{thm_main}
For all $S\in \mathcal{S}_n$, $\mc_3(S)\geq n-2\alpha^*_3(S)$.
\end{theorem}

In Theorem \ref{n-a*_lower} we will prove a stronger statement which, together with Proposition \ref{n-a*_upper}, shows that either $\mc_3(S)=n-\alpha^*_3(S)$, or else the structure of $S$ and the coloring of its edges are restricted in a certain way.  As a corollary of Theorem \ref{n-a*_lower} we recover Theorem \ref{gy_thm}(i).

%\begin{theorem}\label{mc 2/3 lower bound}
%For all $S_n\in \mathcal{S}_n$, 
%\[
%\mc_3(S_n)\geq \frac{2n}{3}+1.
%\]
%\end{theorem}

%\begin{theorem}\label{mc 2/3 upper bound}
%There exists an infinite family of Steiner triple systems $S_{n_k}$ on $n_k$ vertices with 
%\[
%\mc_3(S_{n_k}) \leq (2/3 + o(1))n_k.
%\]
%\end{theorem}

We then use Theorem \ref{thm_main} together with some recent methods developed by Kwan \cite{Kwan} to prove that for almost every $S\in \mathcal{S}_n$, we have $\mc_3(S)=(1-o(1))n$. 

\begin{theorem}\label{thm_random}
There exists a constant $\delta>0$ such that for $n\equiv 1,3\bmod 6$,  if $S\in \mathcal{S}_n$ is chosen uniformly at random, then a.a.s. 
\[
\mc_3(S)\geq n-2n^{1-\delta}.
\]
\end{theorem}

In Corollary \ref{generalupper}, we observe that for all $S\in \mathcal{S}_n$, $\mc_3(S)\leq n-\frac{.417}{3}\sqrt{n\log n}$.  Thus it would be interesting to determine whether the lower bound in Theorem \ref{thm_random} can be improved to $n-O(\sqrt{n\log n})$ (see Problem \ref{prob_aas}).

Additionally, in the process of proving Theorem \ref{thm_random}, we answer a problem raised in \cite[Problem 8]{G2} (see Theorem \ref{discrepancy} and the discussion which follows for more details) by showing that for all sufficiently large $n\equiv 1,3\bmod{n}$, there exists $S\in \mathcal{S}_n$ such that $\alpha^*(S)$ is small (sublinear) and there exists $S\in \mathcal{S}_n$ such that $\alpha^*(S)$ is large (linear).
%Finally, we can observe using Proposition \ref{n-a*_upper} and a result of Phelps and R\"odl \cite{PR} the following upper bound on the value of $\mc_3(S_n)$.
%
%\begin{theorem}
%There exists a constant $c$ such that for all $S_n\in \mathcal{S}_n$, 
%\[
%\mc_3(S_n)\leq n-c\sqrt{n\log n}.
%\]
%\end{theorem}

\subsection{Overview of the paper}

In Section \ref{sec:gen}, we prove some general results giving bounds on $\mc_r(H)$ for $H$ an $r$-uniform hypergraph. We deduce an upper bound on $\mc_3(S)$ for all Steiner triple systems $S$.  We also prove a result which we will later use to deduce lower bounds on $\mc_3(S)$. We additionally give results on $\mc_r(H)$ when $H$ is a random $r$-uniform hypergraph.

In Section \ref{sec:random}, we prove Theorem \ref{thm_random}, that almost all Steiner triple systems $S$ satisfy $\mc_3(S)=(1-o(1))|S|$.

In Section \ref{sec:lower}, we prove an absolute lower bound on $\mc_3(S)$ for all Steiner triple systems $S$, namely that $\mc_3(S) \geq \frac{2|S|}{3} +1$.

In Section \ref{sec:upper}, we prove that there exists an infinite family of Steiner triple systems $\mathcal{S}$ such that $\mc_3(S)=(2/3+o(1))|S|$ for all $S\in \mathcal{S}$.  

In Section \ref{sec:boseskolem}, we discuss bounds on $\mc_3(S)$ for Bose and Skolem triple systems $S$.

Finally, in Section \ref{sec:conc}, we raise a number of Ramsey-type problems in the setting of Steiner triple systems.

\section{General bounds for hypergraphs}\label{sec:gen}

An independent set in a hypergraph $H$ is a set $A\subseteq V(H)$ such for all $e\in E(H)$, $e\not\subseteq A$. 
Given a hypergraph $H$, let $\alpha(H)$ be the maximum integer $a$ such that $H$ contains an independent set of order $a$. Note that for any $k$-uniform hypergraph $H$ we have $\alpha_k^*(H)\geq \floor{\alpha(H)/k}$, since an independent set in $H$ of order $\alpha(H)$ contains $k$ disjoint sets of each of order $\floor{\alpha(H)/k}$.  However, there are hypergraphs $H$ for which $\alpha^*_k(H)$ is arbitrarily larger than $\alpha(H)$.  For instance, take $k$ disjoint sets of size $n/k$ and add all edges except those which touch all $k$ of the sets, then $\alpha(H)=k$, but $\alpha^*_k(H)=n/k$. 

%Note\footnote{LD: Also note that there are hypergraphs $H$ for which $\alpha^*_k(H)$ is arbitrarily larger than $\alpha(H)$.  For instance, take $k$ disjoint sets of size $n/k$ and add all edges except those which touch all $k$ of the sets, then $\alpha(H)=k$, but $\alpha^*_k(H)=n/k$.  Also we can have $\alpha(H)$ much larger than $\alpha^*_k(H)$ (for instance in the Bose construction, we have $\alpha=n/3$, but $\alpha^*_3=2n/9$).} that for all $H$, $\alpha_k^*(H)\geq \floor{\alpha(H)/k}$.  

Note that for all $n\equiv 1,3\bmod 6$ such that $n>3$ and all $S\in \mathcal{S}_n$, we trivially have $\alpha_3^*(S)\geq 1$ and thus $\mc_3(S)\leq n-1$ (c.f. $\mc_3(K^3_n)=n$).  Phelps and R\"odl \cite{PR} showed that there exists a constant $c$ such that for all $S\in \mathcal{S}_n$, $\alpha(S)\geq c\sqrt{n \log n}$ (in fact they proved this for all linear 3-uniform hypergraphs).  The constant was improved in \cite{DLR} and improved further to $.417$ in \cite{KMV}.  Combining this lower bound with Proposition \ref{n-a*_upper} we get the following corollary.

\begin{corollary}\label{generalupper}
For all $S\in \mathcal{S}_n$, $\mc_3(S)\leq n-\alpha_3^*(S)\leq n-\frac{.417}{3}\sqrt{n\log n}$.
\end{corollary}

Grable, Phelps, and R\"odl \cite{GPR} proved that the lower bound on the independence number given above is best possible up to the constant term; that is, there exists an infinite family of Steiner triple systems $\mathcal{S}$ and a constant $C$ such that $\alpha(S)\leq C\sqrt{|S|\log |S|}$, for all $S\in \mathcal{S}$ (in \cite{EV}, it is mentioned that one can take $C=\sqrt{3}$). 

More recently, Linial and Luria \cite[Theorem 1.6]{LL} proved (using different terminology) that exists an infinite family of Steiner triple systems $\mathcal{S}$ such that $\alpha_3^*(S)\leq 100\sqrt{|S|\log |S|}$ for all $S\in \mathcal{S}$.

%\begin{problem}\label{prob_exist}
%Does there exist an infinite family of Steiner triple systems $S_{n_i}$ such that $\alpha_3^*(S_{n_i})=O( \sqrt{n_i\log n_i})$?
%\end{problem}

Given a hypergraph $H$ and $\{u,v\}\in \binom{V(H)}{2}$, let $d(u,v)=|\{S\in E(H): \{u,v\}\subseteq S\}|$ and let $\delta_2(H)=\min\{d(u,v): \{u,v\}\in \binom{V(H)}{2}\}$. We will prove a lower bound on $\mc_3(H)$ for all 3-uniform hypergraphs $H$ with $\delta_2(H)\geq 1$ (i.e. every pair of vertices is contained in at least one edge).  We use the following lemma from \cite{DM} which we reproduce here for completeness (for technical reasons we allow for edges to receive multiple colors, so we have modified the statement of the lemma accordingly).

\begin{lemma}
  \label{lem:3col}
Let $K$ be a complete graph. For every 3-multicoloring of the edges of $K$, either
\begin{enumerate}[label = {\rm(L$_{\arabic*}$)}]
	\item \label{itm:L1} there exists a spanning monochromatic component, or
	\item \label{itm:L2} there exists a partition $\{W, X, Y, Z\}$ of $V(K)$ (all parts non-empty),
such that every edge in $[W,X]$ and $[Y,Z]$ is colored only with blue, every edge in $[W,Y]$ and
$[X,Z]$ is colored only with red, and every edge in $[W,Z]$ and $[X,Y]$ is
colored only with green.
	\item \label{itm:L3} there exists a partition $\{W, X, Y, Z\}$ of $V(K)$ with $X,Y,Z$ non-empty
such that $B:=W\cup X\cup Y$ is connected in blue, $R:=W\cup X\cup Z$ is
connected in red, and $G:=W\cup Y\cup Z$ is connected in green.  Furthermore, every edge in 
$[X,Y]$ is colored only with blue, every edge in $[X,Z]$ is colored only with red, and every edge in $[Y,Z]$ is colored only with green, whereas no edge in $[W,X]$ is green, no edge in $[W,Y]$ is red, and no edge in $[W,Z]$ is blue.
\end{enumerate}
\end{lemma}

\begin{proof}
Suppose $B$ is a monochromatic, say blue, component which is maximal -- in the sense that it is not properly contained in a monochromatic component of any color -- and set $U=V(K)\setminus B$.
If $U=\emptyset$ then we are in case \ref{itm:L1}; so suppose not.  Note that none of the edges from $B$ to $U$ are colored blue.  Let $R$ be a maximal, say red, component which intersects both $B$ and $U$. By the maximality of $B$, we have $B\setminus R\neq \emptyset$. 

First suppose $U\setminus R\neq \emptyset$.  In this case, all of the edges in $[B\cap R,
U\setminus R]$ and $[B\setminus R, U\cap R]$ are colored only with green.  This
implies all of the edges in $[B\cap R, U\cap R]$ and $[B\setminus R, U\setminus R]$ are colored only with 
red, and all of the edges in $[B\cap R, B\setminus R]$ and $[U\cap R, U\setminus R]$ are colored only with 
blue.  So, setting $W:=B\cap R$, $X:=B\setminus R$, $Y:=U\cap R$,
and $Z:=U\setminus R$, we are in case \ref{itm:L2}.

Finally, suppose $U\setminus R=\emptyset$.  In this case the edges in $[B\setminus R, U]$ are colored only with green, so there is a maximal green component $G$ containing $U\cup
(B\setminus R)$.  Then, setting $W:=B\cap R\cap G$, $X:=B\setminus
G$, $Y:=B\setminus R$, and $Z:=U$, we are in case \ref{itm:L3}.
\end{proof}

\begin{theorem}\label{n-a*_lower}
For all $3$-uniform hypergraphs $H$ on $n$ vertices with $\delta_2(H)\geq 1$, either
\begin{enumerate}[label = {\rm(T$_{\arabic*}$)}]
\item\label{itm:T1}  $\mc_3(H)\geq n-\alpha^*_3(H)$, or

\item\label{itm:T2}  there exists a 3-coloring of the edges of $H$ and a partition $\{V_1, V_2, V_3, V_4\}$ of $V(H)$ with $|V_1|\geq |V_2|\geq |V_3|\geq |V_4|$, $\alpha^*_3(H)\geq |V_2|\geq |V_3|\geq |V_4|$, and consequently $|V_1|+|V_i|\geq n-2\alpha_3^*(H)$ for all $i\in \{2,3,4\}$ such that no triple touches three of the sets $V_1, V_2, V_3, V_4$, every triple which touches both $V_1$ and $V_2$ and every triple which touches both $V_3$ and $V_4$ is blue, every triple which touches both $V_1$ and $V_3$ and every triple which touches both $V_2$ and $V_4$ is red, and every triple which touches both $V_1$ and $V_4$ and every triple which touches both $V_2$ and $V_3$ is green.
\end{enumerate}

In particular, for all $3$-uniform hypergraphs $H$ on $n$ vertices with $\delta_2(H)\geq 1$, we have $\mc_3(H)\geq n-2\alpha^*_3(H)$.
\end{theorem}

\begin{proof}
Let $a_3:=\alpha_3^*(H)$.  Given a 3-coloring of the edges of $H$, let $K$ be the complete graph on $V(H)$ with a 3-multicoloring  obtained by coloring $uv\in E(K)$ with color $i$ if and only if there exists an edge $e$ of color $i$ in $H$ such that $\{u,v\}\subseteq e$.  Now apply Lemma \ref{lem:3col} to $K$.  

If we are in case \ref{itm:L1}, then we have a monochromatic component of order $n$ in $H$ which means \ref{itm:T1} holds; so suppose that we are not in case \ref{itm:L1}.

Suppose that we are in case \ref{itm:L2} and without loss of generality, suppose $|W|\geq |X|\geq |Y|\geq |Z|$.  Because of the structure of case \ref{itm:L2}, no edge from $H$ intersects three of the sets $W,X,Y,Z$ which implies $|Y|\leq a_3$.  If $|X|> a_3$, then again since no edge from $H$ intersects three of the sets $W,X,Y,Z$ we must have $|Y\cup Z|\leq a_3$, in which case there is a blue component on $|W|+|X|=n-|Y\cup Z|\geq n-a_3$ vertices which means \ref{itm:T1} holds; so suppose $|X|\leq a_3$.  In this case \ref{itm:T2} holds, as there is a blue component on $|W|+|X|=n-|Y|-|Z|\geq n-2a_3$ vertices, a red component on $|W|+|Y|=n-|X|-|Z|\geq n-2a_3$ vertices, and a green component on $|W|+|Z|=n-|X|-|Y|\geq n-2a_3$ vertices.

Finally, suppose we are in case \ref{itm:L3} and without loss of generality, suppose $|X|\geq |Y|\geq |Z|$.  Because of the structure of case \ref{itm:L3}, no edge from $H$ intersects all three of $X,Y,Z$ which implies $|Z|\leq a_3$.  But now there is a blue component in $H$ on $|W|+|X|+|Y|=n-|Z|\geq n-a_3$ vertices which means \ref{itm:T1} holds.
\end{proof}

Let $H_k(n,p)$ be the binomial random $k$-uniform hypergraph on $n$ vertices where each edge appears independently with probability $p$.  Krivelevich and Sudakov \cite{KS} proved that provided $p=\omega\left(\frac{1}{n^{k-1}}\right)$ and $p=o(1)$, then a.a.s.\
\begin{equation}\label{eqKS}
\alpha(H_k(n,p))= (1+o(1))\left(\frac{k!\log n}{p}\right)^{1/(k-1)}.
\end{equation}
Note that the lower bound is the difficult part of the above estimate, whereas the upper bound is a straightforward first moment calculation.  A similar first moment calculation (c.f. \cite[Section 3.2]{KMV}) shows that a.a.s.\
\begin{equation}\label{eqE}
\alpha_k^*(H_k(n,p))\leq \left(\frac{k\log n}{p}\right)^{1/(k-1)}.  
\end{equation}
And another straightforward calculation shows that if $c>2$ and $p>\frac{c\log n}{n}$, then a.a.s.\ we have $\delta_2(H_3(n,p))\geq 1$.  Thus we obtain the following corollary.  

\begin{corollary}~
\begin{enumerate}
\item If $p=\omega\left(\frac{1}{n^{k-1}}\right)$ and $p=o(1)$, then a.a.s.\ $\mc_k(H_k(n,p))\leq n-\frac{(1-o(1))}{k}\left(\frac{k!\log n}{p}\right)^{1/(k-1)}.$
\item If $c>2$ and $p>\frac{c\log n}{n}$, then a.a.s.\ $\mc_3(H_3(n,p))\geq n-2\left(\frac{3\log n}{p}\right)^{1/2}.$
\end{enumerate}
\end{corollary}

\begin{proof}
\begin{enumerate}
\item We have $\alpha_k^*(H_k(n,p))\geq \floor{\frac{\alpha(H_k(n,p))}{k}}\stackrel{\eqref{eqKS}}{\geq} \frac{(1+o(1))}{k}\left(\frac{k!\log n}{p}\right)^{1/(k-1)}$ a.a.s.\ and thus the result follows from Proposition \ref{n-a*_upper}.

\item By \eqref{eqE}, we have $\alpha_3^*(H_3(n,p))\leq \left(\frac{3\log n}{p}\right)^{1/2}$ a.a.s.\  Also for $c>2$ and $p>\frac{c\log n}{n}$, we have $\delta_2(H_3(n,p))\geq 1$ a.a.s.\  Thus the result follows from Theorem \ref{n-a*_lower}.
\end{enumerate}
\end{proof}

Note that in \cite{BDDE}, the authors prove that if $p=\omega\left(\frac{1}{n^{k-1}}\right)$, then a.a.s.\ $\mc_k(H_k(n,p))\geq (1-o(1))n$. So the above corollary, for certain ranges of $p$, provides quantitative upper bounds for all $k$ and provides a more precise quantitative lower bound in the case $k=3$.

\section{Random Steiner triple systems}\label{sec:random}

In this section we prove Theorem \ref{thm_random}.  That is, we show that with probability approaching 1 as $n\to \infty$, a randomly chosen Steiner triple system $S \in \mathcal{S}_n$ has the property that $\mc_3(S)\geq (1-o(1))n$.  We obtain this as a consequence of Theorem \ref{thm_main} and the following discrepancy-type result.\footnote{We note that in a very recent paper, Ferber and Kwan implicitly prove a more general discrepancy theorem \cite[Theorem 8.1]{FK} which implies $\alpha_3^*(S) = o(n)$ a.a.s.\ for $S \in \mathcal{S}_n$ chosen uniformly at random.  While their discrepancy result is more general, the upper bound on $\alpha^*_3(S)$ which follows from their result would be weaker than in Theorem \ref{discrepancy}.}

\begin{theorem}\label{discrepancy}
There exists a constant $\delta>0$ such that for $n\equiv 1,3\bmod 6$ if $S\in \mathcal{S}_n$ is chosen uniformly at random, then a.a.s.\ 
\[
\alpha_3^*(S) \leq n^{1-\delta}.
\]
\end{theorem}

\begin{corollary}\label{cor:random}
There exists a constant $\delta>0$ such that for $n\equiv 1,3\bmod 6$ if $S\in \mathcal{S}_n$ is chosen uniformly at random, then a.a.s.\ 
\[
\mc_3(S)\geq n-2n^{1-\delta}.
\]
\end{corollary}

Gy\'arf\'as \cite[Problem 8]{G2} asks if $\displaystyle \min_{S\in \mathcal{S}_n} \alpha_3^*(S)$ is significantly smaller than $\displaystyle \max_{S\in \mathcal{S}_n} \alpha_3^*(S)$.  When $n\equiv 3\bmod 18$, Gy\'arf\'as proves that $\displaystyle \max_{S\in \mathcal{S}_n} \alpha_3^*(S)=\frac{n}{3}-1$; furthermore, for all $n\equiv 1,3 \bmod 6$, there exists $S\in \mathcal{S}_n$ with $\alpha^*_3(S)\geq 2\floor{n/9}$ (see Section \ref{sec:boseskolem}). So Theorem \ref{discrepancy} provides an affirmative answer to \cite[Problem 8]{G2} for all sufficiently large $n\equiv 1,3\bmod 6$.

We prove Theorem \ref{discrepancy} using the recent results of Kwan \cite{Kwan} which say that if one can show that a particular property happens with extremely high probability in an appropriately defined random 3-uniform hypergraph, then it also happens with high probability in a randomly chosen Steiner triple system. Before making this precise we need several definitions.

A {\em partial Steiner triple system} is a linear $3$-uniform hypergraph; that is, a $3$-uniform hypergraph in which every pair of vertices is contained in at most one edge.  Let $\mathcal{S}_{n,m}$ be the set of partial systems on $n$ vertices which have $m$ edges. Given a partial Steiner triple system, we may order its edges, and we let $\mathcal{O}_n$ be the set of ordered Steiner triple systems on $n$ vertices and $\mathcal{O}_{n,m}$ be the set of ordered partial Steiner triple systems on $n$ vertices with $m$ edges. Given $S\in \mathcal{O}_{n,m}$ and $i\leq m$ we let $S_i$ be the ordered partial system consisting of the first $i$ edges of $S$. 

Next we define two random processes which we will relate to choosing a random Steiner triple system. 

First, the {\em triangle removal process} is a distribution on $\mathcal{O}_{n,m} \cup \{*\}$. We start with the complete graph $K_n$ and iteratively delete a triangle chosen uniformly from all triangles remaining in the graph. We continue this process until $m$ triangles are removed or there are no more triangles. If the process stops before $m$ triangles are removed, then the output is $``*"$ and otherwise the output is the ordered partial system in $\mathcal{O}_{n,m}$ given by the $m$ deleted triangles in the order they were deleted. Denote this resulting distribution as $\mathbb{R}(n,m)$. 

Second, given an edge probability $p$ let $\mathbb{G}(n,p)$ be the random distribution on $3$-uniform hypergraphs given by independently selecting each triple of $[n]$ with probability $p$. Let $\mathbb{G}^*(n,p)$ be the distribution on partial systems given by choosing a graph from $\mathbb{G}(n,p)$ and deleting all edges which intersect another edge in more than $1$ vertex. 

To analyze what happens in a randomly chosen Steiner triple system, we analyze what happens in $\mathbb{G}(n,p)$ using the following results of Kwan \cite{Kwan}.

\begin{theorem}[Theorem 2.4 of \cite{Kwan} using $\mathcal{Q} = \mathcal{O}_n$ and $\alpha=1/2$]\label{Kwan2.4}
Fixing a sufficiently small $a>0$, there exists $b = b_a > 0$ such that the following holds. Let $\mathcal{P} \subset \mathcal{O}_{n,\frac{1}{6}\binom{n}{2}}$ be a monotone property of ordered partial systems. Let $\mathbf{S} \in \mathcal{O}_n$ be a uniformly random ordered Steiner triple system and let $\mathbf{S'} \in \mathbb{R}(n, \frac{1}{6}\binom{n}{2})$. If $\mathbf{S'} = *$, say that $\mathbf{S'} \not \in \mathcal{P}$. If 
\[
\mathrm{Pr}(\mathbf{S'} \not\in \mathcal{P}) \leq \mathrm{exp}\left(-n^{2-b}\right),
\]
then 
\[
\mathrm{Pr}(\mathbf{S}_{\frac{1}{6}\binom{n}{2}} \not\in \mathcal{P}) \leq \mathrm{exp}\left(-\Omega \left(n^{1-2a}\right)\right).
\]
\end{theorem}

\begin{theorem}[Lemma 2.10 of \cite{Kwan} using $\alpha=1/2$ and $S = \emptyset$]\label{Kwan2.10}
Let $\mathcal{P}$ be a property of unordered partial systems that is monotone increasing (in the sense that $S\in \mathcal{P}$ and $S\subset S'$ implies $S'\in \mathcal{P}$). Let $\mathbf{S} \in \mathbb{R}(n, \frac{1}{6}\binom{n}{2})$ and $\mathbf{S^*} \in \mathbb{G}^*(n, \frac{1}{2n})$. Then 
\[
\mathrm{Pr}(\mathbf{S} \not\in \mathcal{P}) = O\left( \mathrm{Pr}(\mathbf{S^*} \not\in \mathcal{P})\right).
\]
\end{theorem}

The end result of these two theorems is that if we can show that a monotone increasing property holds with probability at least $1 - \mathrm{exp}\left(-n^{2-b}\right)$ in the distribution $\mathbb{G}^*(n, \frac{1}{2n})$ then it will also happen with probability tending to $1$ in a randomly chosen Steiner triple system. 
In particular, we will use the following lemma from \cite{Kwan} to show that $\alpha_3^*(\mathbb{G}^*(n, \frac{1}{2n}))\leq n^{1-\delta}$.  As pointed out in \cite{FK}, the following lemma can also be derived from \cite[Theorem 1.3]{W}.

\begin{lemma}[{\cite[Lemma 2.11]{Kwan}} and {\cite[Theorem 1.3]{W}}]\label{concentration}
Let $\omega=(\omega_1, \dots, \omega_N)$ be a sequence of independent, identically distributed random variables with $\mathrm{Pr}(\omega_i=1)=p$ and $\mathrm{Pr}(\omega_i=0)=1-p$.  Let $f:\{0,1\}^N\to \mathbb{R}$ satisfy the Lipschitz condition $|f(\omega)-f(\omega')|\leq K$ for all pairs $\omega, \omega'\in \{0,1\}^N$ differing in exactly one coordinate.  Then 
\[
\mathrm{Pr}(|f(\omega)-\mathbb{E}f(\omega)|>t)\leq \exp\left(-\frac{t^2}{4K^2Np+2Kt}\right).
\]
\end{lemma}

\begin{proof}[Proof of Theorem \ref{discrepancy}]
Let $A,B,C$ be fixed disjoint subsets of $[n]$ of size $n^{1-\delta}$ where $\delta$ is a small positive constant that will be chosen later. Let $\mathcal{P}_{A,B,C}$ be the property that exists a triple which intersects all of $A$, $B$, and $C$.  Note that $\mathcal{P}_{A,B,C}$ is monotone increasing.

First we estimate the probability that a fixed set of three vertices appears as an edge in the distribution $\mathbb{G}^*(n, \frac{1}{2n})$. Let $x,y,z$ be fixed vertices. Then 
\[
\mathrm{Pr}(xyz \mbox{ forms an edge}) = \frac{1}{2n} \left(1-\frac{1}{2n}\right)^{3(n-3)} \sim \frac{1}{2n}e^{-3/2}.
\]
Therefore, the expected number of edges induced by $A,B,C$ in $\mathbb{G}^*(n, \frac{1}{2n})$, which we denote by $e^*(A,B,C)$ is
\[
\mathbb{E}(e^*(A,B,C)) = |A||B||C|\mathrm{Pr}(xyz \mbox{ forms an edge}) \sim \frac{e^{-3/2}}{2}n^{2-3\delta}.
\]

Let the triples on $[n]$ be ordered arbitrarily and let $Z_1 ,\cdots, Z_{\binom{n}{3}}$ indicator random variables where $Z_i = 1$ indicates that the $i$'th triple appears in $\mathbb{G}(n, p)$.  Let $f_{A,B,C}$ be the function where $f_{A,B,C}\left(Z_1 ,\cdots, Z_{\binom{n}{3}}\right)$ equals the number of edges in the resulting output of $\mathbb{G}^*(n,p)$ each with one endpoint in each of $A$, $B$, and $C$. Note that $f_{A,B,C}$ is $3$-Lipschitz, since changing one triple in $\mathbb{G}(n,p)$ may add at most one edge or remove at most three edges from the resulting output of $\mathbb{G}^*(n,p)$. 
Thus $f_{A,B,C}$ satisfies the hypotheses of Lemma \ref{concentration} (with $K=3$). 

Now using Lemma \ref{concentration} with $t=\mathbb{E}(e^*(A,B,C))\sim \frac{e^{-3/2}}{2}n^{2-3\delta}$ and the fact that $$\mathrm{Pr}\left(\mathbb{G}^*\left(n, \frac{1}{2n}\right)\not\in \mathcal{P}_{A,B,C}\right)=\mathrm{Pr}\left(e^*(A,B,C)=0\right)=\mathrm{Pr}\left(f_{A,B,C}\left(Z_1 ,\cdots, Z_{\binom{n}{3}}\right)\leq  0\right),$$ we have 
\[
\mathrm{Pr}\left(\mathbb{G}^*\left(n, \frac{1}{2n}\right)\not\in \mathcal{P}_{A,B,C}\right)\leq \exp\left(\frac{-t^2}{3n^2+6t}\right)\leq \exp\left(-\Omega\left(n^{2-6\delta}\right)\right).
\]

Now let $\mathcal{P}$ be the property that for all disjoint sets $A$, $B$, and $C$ each of size $n^{1-\delta}$ there is at least one edge which touches all three sets. Since there are at most $2^n$ choices for each of these sets, by the union bound we have 
\[
\mathrm{Pr}\left(\mathbb{G}^*\left(n,\frac{1}{2n}\right) \not\in \mathcal{P}\right) \leq 2^{3n} \cdot \exp\left(-\Omega\left(n^{2-6\delta}\right)\right) = \exp\left(-\Omega\left(n^{2-6\delta}\right)\right).
\]

Therefore, letting $b$ be the constant given by Theorem \ref{Kwan2.4} and $\delta$ any positive constant less than $b/6$, we have by Theorems \ref{Kwan2.4} and \ref{Kwan2.10} that if $\mathbf{S} \in \mathcal{O}_n$ is a randomly chosen ordered Steiner triple system, then the first half of its edges induce at least $1$ edge on any $3$ disjoint sets $A,B,C$ of size $n^{1-\delta}$ almost surely, and so $\alpha_3^*(\mathbf{S}) \leq n^{1-\delta}$ almost surely.
\end{proof}

We close this section by raising the following problem, part (i) would strongly imply the existence result of Linial and Luria \cite{LL} and part (ii) would strongly imply the existence result of Grable, Phelps, and R\"odl \cite{GPR} (both of which are discussed in Section \ref{sec:gen}).  

\begin{problem}\label{prob_aas}
Is it true that for $n\equiv 1,3\bmod 6$, if $S\in \mathcal{S}_n$ is chosen uniformly at random, then a.a.s. 
\begin{enumerate}
\item $\alpha_3^*(S)=O(\sqrt{n\log n})$ and
\item $\alpha(S)=O(\sqrt{n\log n})$?
\end{enumerate}
\end{problem}

We note that a positive answer to part (i) would imply a positive answer to part (ii) because, as previously observed, we have $\floor{\frac{\alpha(S)}{3}}\leq \alpha^*_3(S)$.

\section{A lower bound for all Steiner triple systems}\label{sec:lower}

We begin with an upper bound on the 3-partite hole number of Steiner triple systems (which slightly improves the trivial upper bound).  This was previously observed by Gy\'arf\'as in the introduction of \cite{G2}, but in order to keep the paper self-contained and to make the calculation explicit, we provide a proof.

\begin{proposition}\label{alpha3_upper}
For all $S\in \mathcal{S}_n$, $\alpha^*_3(S)\leq \frac{n}{3}-1$.
\end{proposition}

\begin{proof}
Let $S \in \mathcal{S}_n$ and suppose for contradiction that $\alpha_3^*(S) \geq \frac{n-2}{3}$.  Let $V_1, V_2, V_3\subseteq [n]$ be disjoint sets with $|V_1|, |V_2|, |V_3| \geq \frac{n-2}{3}$ such that no triple in $S$ touches all three of $V_1, V_2, V_3$.  

If $n\equiv 3\bmod 6$, then $|V_1|, |V_2|, |V_3| \geq \frac{n-2}{3}$ implies $|V_1|=|V_2|=|V_3|=\frac{n}{3}$.  Since each triple touches at most two of the parts, we have 
\begin{equation}\label{3part}
\frac{n(n-1)}{6}=e(S)\leq \binom{|V_1|}{2} + \binom{|V_2|}{2} + \binom{|V_3|}{2}=3\binom{n/3}{2}=\frac{n(n-3)}{6},
\end{equation}
a contradiction.

So suppose $n\equiv 1\bmod 6$, in which case $|V_1|, |V_2|, |V_3| \geq \frac{n-2}{3}$ implies $|V_1|,|V_2|,|V_3|\geq \frac{n-1}{3}$.  So let $\{\{v\}, V_1, V_2, V_3\}$ be a partition of $[n]$ with $|V_1|=|V_2|=|V_3|=\frac{n-1}{3}$ such that no triple touches all three of $V_1, V_2, V_3$.  In this case, we have 
\begin{align*}
\frac{n(n-1)}{6}=e(S)&\leq \binom{|V_1|}{2} + \binom{|V_2|}{2} + \binom{|V_3|}{2}+\frac{n-1}{2}\\
&= 3\binom{(n-1)/3}{2} +\frac{n-1}{2} =\frac{(n-1)(n-4)}{6}+\frac{n-1}{2}=\frac{(n-1)^2}{6},
\end{align*}
again a contradiction.
\end{proof}

Now we use Proposition \ref{alpha3_upper} and Theorem \ref{n-a*_lower} to prove a general lower bound on the size of a largest monochromatic component in every 3-coloring of every $S\in \mathcal{S}_n$.  Again, this was previously proved by Gy\'arf\'as \cite[Theorem 1]{G2}, but since our proof is different, we include it here. 

\begin{theorem}\label{mc 2/3 lower bound}
For all $S\in \mathcal{S}_n$, $\mc_3(S)\geq \frac{2n}{3}+1$.
\end{theorem}

\begin{proof}
Let $S \in \mathcal{S}_n$. Note that since $n$ is congruent to either $1$ or $3$ mod $6$ and $\mc_3(S)$ is an integer, in order to show that $\mc_3(S) \geq \frac{2n}{3}+1$ it suffices to show that $\mc_3(S) > \frac{2n+1}{3}$.  

By Proposition \ref{alpha3_upper} we have $\alpha^*_3(S)\leq \frac{n}{3}-1$.  Now we apply Theorem \ref{n-a*_lower} and either get $\mc_3(S)\geq n-\alpha^*_3(S)\geq \frac{2n}{3}+1$, or 3-coloring of $S$ and a partition $\{V_1, V_2, V_3, V_4\}$ of $V(S)$ satisfying \ref{itm:T2}.  If we are in this situation, then $V_1\cup V_2$ induces a monochromatic component and thus $\mc_3(S) \geq |V_1| + |V_2|$.  Since no edge intersects three of the sets $V_1,V_2,V_3,V_4$ we have 

\[
\frac{n(n-1)}{6}=e(S)\leq \sum_{i=1}^4\binom{|V_i|}{2}
\]

This implies that $\mc_3(S) \geq z_1$ where $z_1$ is the solution to the following integer program.

\begin{align*}
& \mbox{minimize} &z_1 = f(x_1, x_2, x_3, x_4) = x_1+x_2\\
&\mbox{subject to} &x_1+x_2+x_3+x_4 = n\\
& & x_1 \geq x_2 \geq x_3 \geq x_4\\
& & \frac{n(n-1)}{6} \leq \sum_{i=1}^4 \binom{x_i}{2}.
\end{align*}

Instead we solve a slightly more relaxed integer program:

\begin{align*}
& \mbox{minimize} &z_2 = g(x_1, x_2) = x_1+x_2\\
&\mbox{subject to} &x_1+3x_2 \geq  n\\
& & x_1 \geq x_2 \\
& & \frac{n(n-1)}{6} \leq \binom{x_1}{2} + 3\binom{x_2}{2}.
\end{align*}

Note that $z_2 \leq z_1$ since for any $(x_1,x_2, x_3, x_4)$ which is feasible for the first problem we have $(x_1, x_2)$ is feasible for the second problem, and $f(x_1,x_2, x_3, x_4) = g(x_1, x_2)$. Therefore, $\mc_3(S) \geq z_2$ and it suffices to show that $z_2 > \frac{2n+1}{3}$ for $n\geq 3$.

To show this, assume that $(x_1, x_2)$ is an optimal solution for the second problem. Then we have $x_1 \geq n-3x_2$ and so 
\[
\frac{n(n-1)}{6} \leq \binom{n-3x_2}{2} + 3\binom{x_2}{2}.
\]
This implies $18x_2^2 - 9nx_2 + n(n-1) \geq 0$, which implies that either 
\[
x_2 \geq \frac{9n + \sqrt{9n^2 + 72n}}{36}.
\]
or 
\[
x_2 \leq  \frac{9n - \sqrt{9n^2 + 72n}}{36}.
\]
If $x_2 \geq \frac{9n + \sqrt{9n^2 + 72n}}{36}$, then $z_2=x_1+x_2\geq 2\left(\frac{9n + \sqrt{9n^2 + 72n}}{36}\right)= \frac{n}{2} + \frac{n}{6}\sqrt{1 + \frac{8}{n}}$.  If $x_2 \leq  \frac{9n - \sqrt{9n^2 + 72n}}{36}$, then $z_2 = x_1 + x_2 \geq n-2x_2 \geq \frac{n}{2} + \frac{n}{6}\sqrt{1+ \frac{8}{n}}$.  Since $\frac{n}{2} + \frac{n}{6}\sqrt{1+ \frac{8}{n}}>\frac{2n+1}{3}$ for all $n \geq 2$, we are done in either case.

%Note that this is strictly larger than $\frac{2n}{3}$ and so we have $\mc_3(S_n) \geq \frac{2n+1}{3}$. Using the Taylor series for $\sqrt{1+x}$ shows that for $n\geq 16$ we have 
%\[
%z_2 \geq \frac{n}{2} + \frac{n}{6} \left(1 + \frac{4}{n} - \frac{8}{n^2}\right) > \frac{2n+1}{3}.
%\]

\end{proof}

\section{An upper bound for infinitely many Steiner triple systems}\label{sec:upper}

In this section we prove the following result which shows that Theorem \ref{mc 2/3 lower bound} is asymptotically best possible.  

\begin{theorem}\label{mc 2/3 upper bound}
There exists an infinite family of Steiner triple systems $S_{n_k}$ on $n_k$ vertices with 
\[
\mc_3(S_{n_k}) \leq (2/3 + o(1))n_k.
\]
\end{theorem}

As mentioned in the introduction, a stronger (non-asymptotic) statement was proved by Gy\'arf\'as \cite[Theorem 4]{G2}, which shows that Theorem \ref{mc 2/3 lower bound} is actually best possible for certain values of $n$.  Regardless, we provide a proof of Theorem \ref{mc 2/3 upper bound} as it turns out that Theorem \ref{mc 2/3 upper bound} was alluded to in the abstract and introduction of \cite{G2}; however, as far as we know, a proof of this fact doesn't appear in the literature.

We begin with the following definition from \cite{CDR}. Given a Steiner triple system $(V, \mathcal{B})$, a $3$-coloring of the vertex set $\phi:V\to \{1,2,3\}$ is called a {\em bicoloring} if for all $B\in \mathcal{B}$ we have 
\[
\left| \bigcup_{v\in B} \phi(v)\right|  = 2;
\]
that is, every triple contains exactly $2$ colors. We say that a Steiner triple system is $(a,b,c)$-bicolorable if there is a bicoloring $\phi$ with 
\[
|\phi^{-1}(1)| = a \quad |\phi^{-1}(2)| = b \quad |\phi^{-1}(3)| = c.
\]

In \cite{CDR} the following recursive construction of bicolorings is given.

\begin{theorem}[\cite{CDR} Theorem 2.4]\label{CDR recursive construction}
If there exists an $(a,b,c)$-bicolorable Steiner triple system with $c = \max\{a,b,c\}$ and $c\leq a+b$, and if there exists an $(x,y,z)$-bicolorable Steiner triple system, then there exists an $(ay+bz+cx, az+bx+cy, ax+by+cz)$-bicolorable Steiner triple system.
\end{theorem}

We use bicolorable Steiner triple systems to produce edge colorings with small monochromatic components.

\begin{lemma}\label{bicoloring to edge coloring}
If $S \in \mathcal{S}_n$ is $(a,b,c)$-bicolorable with $a\leq b\leq c$, then $\mc_3(S) \leq b+c$.
\end{lemma}
\begin{proof}
%Let $S_n = (V, \mathcal{B})$. Given the vertex bicoloring $\phi$, we construct an edge coloring $\chi:\mathcal{B} \to \{1,2,3\}$ given by 
%\[
%\chi(B) = \{1,2,3\} \setminus \left( \bigcup_{v\in B} \phi(v)\right).
%\]
%Since $\phi$ is a bicoloring, $\chi$ is well-defined. Furthermore, for $\{i,j,k\} = \{1,2,3\}$, an edge colored with color $i$ is contained in the set $\phi^{-1}(j) \cup \phi^{-1}(k)$, and so the largest color class (and therefore the largest monochromatic component) spans $b+c$ vertices.
Since $S$ is $(a,b,c)$-bicolorable, we have that $\alpha^*_3(S)\geq a$ and thus by Theorem \ref{n-a*_upper}, we have $\mc_3(S)\leq n-a=b+c$.
\end{proof}

Theorem \ref{mc 2/3 upper bound} now follows immediately as a corollary from Lemma \ref{bicoloring to edge coloring} and the following proposition.

\begin{proposition}
For each $k\in \mathbb{N}$ there exists a $(M_k, M_k, N_k)$-bicolorable Steiner triple system where $M_k, N_k \to \infty$ and $\frac{M_k}{N_k}\to 1$ as $k\to \infty$.
\end{proposition}

\begin{proof}
It is easy to see that the unique Steiner triple system on 9 vertices is a $(1,4,4)$-bicolorable Steiner triple system.  Applying Theorem \ref{CDR recursive construction} with $a=x=1$ and $b=c=y=z=4$ gives a $(24, 24, 33)$-bicolorable Steiner triple system.  

Let $M_0 = 24$ and $N_0 = 33$ and for all $k\in \mathbb{N}$ let 
\begin{align*}
M_k &= M_{k-1}^2 + 2M_{k-1}N_{k-1} ~\text{ and}\\
N_k &= 2M_{k-1}^2+N_{k-1}^2.
\end{align*}
By induction, we have $M_k\leq N_k\leq 2M_k$ for all $k$, since
\begin{align*}
M_{k-1}^2+2M_{k-1}N_{k-1}\leq  2M_{k-1}^2+N_{k-1}^2\leq 2M_{k-1}^2+(2M_{k-1})^2\leq 2(M_{k-1}^2+2M_{k-1}N_{k-1}).
\end{align*}

Therefore, we may apply Theorem \ref{CDR recursive construction} with $a=b=x=y=M_{k-1}$ and $c=z=N_{k-1}$ and have that for each $k$ there is a $(M_k, M_k, N_k)$-bicolorable Steiner triple system. It is clear that $M_k$ and $N_k$ go to infinity, and it remains to show that their ratio tends to $1$. Define 
\[
r_k = \frac{M_k}{N_k},
\]
and note that $r_k \leq 1$ for all $k$. Also note that 
\begin{align}\label{r equation}
r_k = \frac{M_{k-1}^2 + 2M_{k-1}N_{k-1}}{2M_{k-1}^2 + N_{k-1}^2} = \frac{r_{k-1}^2N_{k-1}^2 + 2r_{k-1}N_{k-1}^2}{2r_{k-1}^2N_{k-1}^2 + N_{k-1}^2} =r_{k-1}\left(\frac{r_{k-1}+2}{2r_{k-1}^2+1} \right).
\end{align}
Now since $0\leq r_k \leq 1$ we have $2r_{k-1}^2+1\leq r_{k-1}^2+2\leq r_{k-1}+2$ which implies $$r_{k-1}\leq r_{k-1}\left(\frac{r_{k-1}+2}{2r_{k-1}^2+1}\right) = r_k$$ by \eqref{r equation}. 
Since $r_k$ is nondecreasing and bounded above, there is some $r$ such that $r_k\to r$ as $k\to \infty$. By \eqref{r equation}, we have
\[
r =\lim_{k\to\infty}r_k=\lim_{k\to \infty} \frac{r_{k-1}^2 + 2r_{k-1}}{2r_{k-1}^2 + 1}=r\left(\frac{r+2}{2r^2+1}\right),
\]
which is satisfied if and only if $r\in \{-1/2,0, 1\}$.  Since $r_0 >0$ we must have $r=1$.
\end{proof}

\section{Bose and Skolem triple systems}\label{sec:boseskolem}

There are unique Steiner triple systems $S_7$ and $S_9$ on 7 and 9 vertices respectively and it is straightforward to directly show that $mc_3(S_7)=6$ and $mc_3(S_9)=7$ (and $\alpha^*_3(S_7)=1$ and $\alpha^*_3(S_9)=2$).  These triple systems are the first non-trivial cases of two common constructions of Steiner triple systems by Bose and Skolem.  A {\em quasigroup} is a pair $(Q, \circ)$ where $Q$ is a set of size $k$ and $\circ$ is a binary operation on $Q$ such that the multiplication table is a latin square. That is, the equations $a\circ x = b$ and $y\circ a = b$ each have a unique solution for all $a,b\in Q$. A quasigroup is called {\em idempotent} if $a\circ a = a$ for all $a\in Q$ and is called {\em commutative} if $a\circ b = b\circ a $ for all $a,b\in Q$. A quasigroup is called {\em half-idempotent} if $k$ is even and $Q$ can be ordered in a way that cells $(i,i)$ and $(k+i, k+i)$ contain $i$ in the multiplication table for all $1\leq i\leq k$.

\medskip
\noindent{\bf Bose construction:} Let $n = 6k+3$ and let $(Q, \circ)$ be a commutative, idempotent quasigroup of order $2k+1$. Define $(V, \mathcal{B})$ as $V = Q \times \{0,1,2\}$ and $\mathcal{B}$ contains two types of triples.

Type 1: For all $a\in Q$ let $\{(a,0), (a,1), (a,2)\}$ be a triple.

Type 2: For all distinct $a,b\in Q$ and $i\in \{0,1,2\}$ let $\{(a,i), (b,i), (a\circ b, i+1)\}$ be a triple where addition in the second coordinate is done mod $3$. 

\noindent Because $Q$ is commutative and idempotent, one can check that this forms a Steiner triple system on $6k+3$ vertices (see \cite{CR} for more details).

Let $n = 6k+3$ and let $S$ be a Bose triple system on $n$ points. Then we may color the Type 1 triples with $3$ colors as evenly as possible, and color the Type 2 triples $\{(a,i), (b,i), (a\circ b, i+1)\}$ with color $i-1$ mod $3$. Then each color will touch at most $4k+2 + \left \lceil \frac{2k+1}{3}\right\rceil \sim \frac{7n}{9}$ vertices, and so we have $\mc_3(S) \lesssim \frac{7n}{9}$.

\medskip
\noindent{\bf Skolem construction:} Let $n = 6k+1$ and let $(Q, \circ)$ be a half-idempotent commutative quasigroup with $Q = \{0,\cdots, 2k-1\}$. Define $(V,\mathcal{B})$ by $V = \{\infty\} \cup (Q\times \{0,1,2\})$ and $\mathcal{B}$ as triples of three types.

Type 1: For all $0\leq a\leq k-1$, the triple $\{(a,0), (a,1), (a,2)\}$.

Type 2: For all $0\leq a\leq k-1$ the triples $\{\infty, (k\circ a, 0), (k, 1)\}$, $\{\infty, (k\circ a, 1), (k,2)\}$, and $\{\infty, (k\circ a, 2), (k, 0)\}$

Type 3: For all distinct $a,b\in Q$ and $i \in \{0,1,2\}$, the triple $\{(a,i), (b,i), (a\circ b, i+1)\}$ with addition in the second coordinate done mod $3$. 

\noindent One may check that this forms a Steiner triple system on $6k+1$ vertices  (see \cite{CR} for more details).

Given $S$, a Skolem triple system on $6k+1$ vertices, we may color Type 1 triples as evenly as possible with $3$ colors. Then we may color Types 2 and 3 triples with the color that is not represented in the second coordinate in any vertex of the triple. Each color class touches at most $\left \lceil \frac{k}{3}\right\rceil + 4k +1 \sim \frac{13n}{18}$ vertices, and so $\mc_3(S) \lesssim \frac{13n}{18}$.

\medskip

It seems likely that these colorings are best-possible for some or maybe ``most" Bose and Skolem triple systems, so we pose the following problems.

\begin{problem}\label{bose problem}
Is there an infinite sequence $\mathcal{S}$ of Bose triple systems such that $\mc_3(S)\sim 7|S|/9$ for all $S\in \mathcal{S}$?

\end{problem}

\begin{problem}\label{skolem problem}  

Is there an infinite sequence $\mathcal{S}$ of Skolem triple systems such that $\mc_3(S)\sim 13|S|/18$ for all $S\in \mathcal{S}$?

\end{problem}

In light of Theorem \ref{n-a*_lower}, in order to address say Problem \ref{bose problem} we can try to show that the following is impossible:  Given a Bose triple system $S=(V, \mathcal{B})$ on $n$ points, there is a partition of $V$ into four sets with $|V_1|\geq |V_2|\geq |V_3|\geq |V_4|$ and $\frac{2n}{9}\sim \alpha^*_3(S)\geq |V_2|\geq |V_3|\geq |V_4|$ and no triple touches three of the sets $V_1, V_2, V_3, V_4$.  There are at least two possible approaches here.  First, one can define an explicit quasigroup that defines the Bose triple system. For example, for $2k+1$ prime we may define the quasigroup by $Q = \mathbb{Z}_{2k+1}$ and $a\circ b = \frac{1}{2}(a+b)$ where addition is done in the field. In \cite{BN}, this Steiner triple system was shown to have nice expansion properties. Second, one could define a quasigroup ``randomly". Let $K$ be the complete graph on $[2k+1]$. Then we may associate multiplication tables of a quasigroup of order $2k+1$ to decompositions of $K$ into matchings of size $k$; that is, matchings which cover all but one vertex of $K$ (For each $x\in [2k+1]$, let $M_x$ be the perfect matching which is not incident with $x$. Place an $x$ in the $i$th row and $j$th column and the $j$th row and $i$th column for all $\{i,j\}\in M_x$).  Choosing such a decomposition randomly (or choosing some amount of nearly perfect matchings randomly and then completing the decomposition) should yield a quasigroup with nice expansion properties. 

\section{Conclusion}\label{sec:conc}

We begin this section by drawing the reader's attention to the following conjecture which is best possible if true by Theorem \ref{gy_thm}(i) (Theorem \ref{mc 2/3 lower bound}) and Proposition \ref{alpha3_upper}, and is known to be true  when $n\equiv 3\bmod 18$ by Theorem \ref{gy_thm}(iii).

\begin{conjecture}[Gy\'arf\'as {\cite[Conjecture 2, Conjecture 5]{G2}}]\label{conj:n/3}
For all $n\equiv 1,3 \bmod 6$, 
\begin{enumerate}
\item $\displaystyle \max_{S\in \mathcal{S}_n} \alpha^*_3(S)=\left\lfloor\frac{n}{3}\right\rfloor-1$ and

\item $\displaystyle \min_{S\in \mathcal{S}_n} \mc_3(S)=\left\lceil\frac{2n}{3}\right\rceil+1$.
\end{enumerate}
\end{conjecture}

We note that a positive answer to Conjecture \ref{conj:n/3}(i) would imply a positive answer to Conjecture \ref{conj:n/3}(ii).

If we let $\mathcal{C}$ be the set of real numbers $c$ such that there exists a sequence of Steiner triple systems $S_{n_k}$ on $n_k$ vertices with 
\[
\lim_{k\to \infty} \frac{\mc_3(S_{n_k})}{n_k} = c,
\]
then the results in this paper together with the results in \cite{G2} show that $\{2/3, 1\} \subseteq \mathcal{C}$. Answering Problems \ref{bose problem} and \ref{skolem problem} affirmatively would show that $\frac{7}{9}$ and $\frac{13}{18}$ are also in $\mathcal{C}$. It would be interesting to determine other real numbers in $\mathcal{C}$.

By combining Theorem \ref{n-a*_upper} and Theorem \ref{n-a*_lower}, we have for all $S\in \mathcal{S}_n$,  $\mc_3(S)=n-\alpha^*_3(S)$ unless there is a partition $\{V_1, V_2, V_3, V_4\}$ of $V(H)$ with $|V_1|\geq |V_2|\geq |V_3|\geq |V_4|$,  $\alpha^*_3(H)\geq |V_2|\geq |V_3|\geq |V_4|$, and $|V_1|+|V_2|<n-\alpha^*_3(S)$ such that no edge of $H$ intersects three of the sets $V_1,V_2,V_3,V_4$.  This raises the question of determining some sufficient condition that a Steiner triple system can satisfy which would rule out this latter possibility.

\begin{problem}
Does there exist an infinite family of Steiner triple systems $\mathcal{S}$ such that $\alpha^*_3(S)=(\frac{1}{6}+o(1))|S|$ and $\mc_3(S)=(2/3+o(1))|S|$ for all $S\in \mathcal{S}$?  

More generally, does there exist an infinite family of Steiner triple systems $\mathcal{S}$ such that $\mc_3(S)=|S|-2(1-o(1))\alpha^*_3(S)$?
\end{problem}

At the moment, we don't even have an example of a Steiner triple system $S\in \mathcal{S}_n$ such that $\mc_3(S)<n-\alpha^*_3(S)$, although we suspect they exist.

Gy\'arf\'as \cite{G} and Gy\'arf\'as and Haxell \cite{GH} proved $\mc_4(K_n^3)\geq 3n/4$, $\mc_5(K_n^3)\geq 5n/7$, $\mc_6(K_n^3)\geq 4n/6$ and these bounds are tight when $n$ is divisible by $4$, $7$, and $6$ respectively.  In general, F\"uredi and Gy\'arf\'as \cite{FG} showed that $\mc_r(K_n^3)\geq n/q$ where $q$ is the smallest integer such that $r\leq q^2+q+1$ and this is sharp when $q^3$ divides $n$, $r=q^{2}+q+1$, and an affine space of dimension $3$ and order $q$ exists.  As mentioned in the introduction, for $S\in \mathcal{S}_n$ we have $\mc_r(S)\geq \mc_r(K_n)\geq \frac{n}{r-1}$.  Surprisingly, Gy\'arf\'as \cite[Proposition 7]{G2} proved that if $r-1\equiv 1,3 \bmod 6$ and $r-1$ is a prime power, then for infinitely many $n$, $\displaystyle \min_{S\in \mathcal{S}_n}\mc_r(S)=\frac{n}{r-1}$.  It would be interesting to extend Theorem \ref{thm_main} to more colors, where a main obstacle would be a generalization of Lemma \ref{lem:3col}.  Likewise, it would be interesting to extend Theorem \ref{thm_random} to more colors.  

\begin{problem}
Let $S\in \mathcal{S}_n$ and let $r\geq 4$.  
\begin{enumerate}
\item Determine bounds on $\mc_r(S)$ in terms of $\alpha^*_3(S)$ (or some related parameter).  
\item Determine a lower bound on $\mc_r(S)$ which holds a.a.s.\ when $S\in \mathcal{S}_n$ is chosen uniformly at random.  
\end{enumerate}
\end{problem}

It is known (see \cite{AF}, \cite{AFL}) that in any $2$-coloring of the edges of $K_n^3$, there is a monochromatic matching on at least $3n/4$ vertices and a monochromatic loose cycle (and consequently a loose path) on  $(4/5-o(1))n$ vertices \cite{HLPRRSS}.  We propose studying these problems in the setting of Steiner triple systems.  As we did for monochromatic components, it would be interesting to get an absolute lower bound for all $S\in \mathcal{S}_n$, an upper bound for an infinite family $S_{n_k}$, and also to consider the case of a uniformly random $S\in \mathcal{S}_n$. 

\begin{problem}
Let $S\in \mathcal{S}_n$.  For an arbitrary 2-coloring of the edges of $S$,
\begin{enumerate}
\item what is the largest monochromatic matching?
\item what is the longest monochromatic loose path/cycle?
\end{enumerate}
\end{problem}

\section*{Acknowledgements}
We thank Tom Bohman, Charlie Colbourn, and Asaf Ferber for helpful discussions.  We also thank the referees for their  constructive feedback.


\begin{thebibliography}{9}
\bibitem{AF} N. Alon and P. Frankl. Families in which disjoint sets have large union. \emph{Annals of the New York Academy of Sciences} \tbf{555}, no. 1 (1989), 9--16.

\bibitem{AFL} N. Alon, P. Frankl, and L. Lov\'asz. The chromatic number of Kneser hypergraphs. \emph{Transactions of the American Mathematical Society} \tbf{298}, no. 1 (1986), 359--370.


\bibitem{BDDE} P.\ Bennett, L.\ DeBiasio, A. Dudek, and S. English. Large monochromatic components and long monochromatic cycles in random hypergraphs. \emph{European Journal of Combinatorics} \tbf{76} (2019), 123--137.

\bibitem{BN} Z. Bl\'azsik and Z. Nagy, Spreading linear triple systems and expander triple systems, arXiv:1906.03149

\bibitem{CDR} C.\ J.\ Colbourn, J.\ H.\ Dinitz, and A.\ Rosa. Bicoloring Steiner triple systems. \emph{The Electronic Journal of Combinatorics} \tbf{6}, no. 1 (1999), P.25.

\bibitem{CR} C.J. Colbourn and A. Rosa. Triple systems. \emph{Oxford University Press}, 1999.

\bibitem{DLR} R. Duke, H. Lefmann and V. R\"odl. On uncrowded hypergraphs. \emph{Random Structures
and Algorithms} \tbf{6}, (1995), 209--212.

\bibitem{DM} L. DeBiasio, P. McKenney. Density of monochromatic infinite subgraphs. \emph{Combinatorica}, (2019). https://doi.org/10.1007/s00493-018-3724-2

%\bibitem{Freedman} D.\ A.\ Freedman. On tail probabilities for martingales, {\em Annals of Probability} (1975), 100--118.

\bibitem{EV} A. Eustis and J. Verstra\"ete. On the independence number of Steiner systems. \emph{Combinatorics, Probability and Computing} \tbf{22}, no. 2 (2013), 241--252.

\bibitem{FK} A.\ Ferber and M.\ Kwan, Almost all Steiner triple systems are almost resolvable, arXiv:1907.06744

\bibitem{FG} Z.\ F\"uredi and A.\ Gy\'arf\'as. Covering $t$-element sets by partitions. \emph{European Journal of Combinatorics} \emph{12}, no. 6 (1991), 483--489.

\bibitem{GPR} D.A. Grable, K. T. Phelps, and V. R\"odl. The minimum independence number for designs. \emph{Combinatorica} \tbf{15}, no. 2 (1995), 175--185.

\bibitem{G} A.\ Gy\'arf\'as. Partici\'ofed\'esek \'es lefog\'ohalmazok hipergr\'afokban, Tanulm\'anyok-MTA Sz\'amit\'astechn. Automat. Kutat\'o Int. Budapest \tbf{62}, (1977).

\bibitem{G2} A.\ Gy\'arf\'as. Large cross-free sets in Steiner triple systems. \emph{Journal of Combinatorial Designs} \tbf{23}, no. 8 (2015), 321--327.

\bibitem{GH} A.\ Gy\'arf\'as and P. Haxell. Large monochromatic components in colorings of complete 3-uniform hypergraphs.  \emph{Discrete Mathematics} \tbf{309}, no. 10 (2009), 3156--3160.


\bibitem{HLPRRSS} P.\ E.\ Haxell, T.\ {\L}uczak, Y.\ Peng, V.\ R\"odl, A.\ Ruci\'nski, M.\ Simonovits, and J.\ Skokan. The Ramsey number for hypergraph cycles I. \emph{Journal of Combinatorial Theory, Series A} \tbf{113}, no. 1 (2006), 67--83.

\bibitem{KMV} A.\ Kostochka, D.\ Mubayi, and J.\ Verstra\"ete. On independent sets in hypergraphs. \emph{Random Structures \& Algorithms} \tbf{44}, no. 2 (2014), 224--239.

\bibitem{KS} M.\ Krivelevich and B.\ Sudakov. The chromatic numbers of random hypergraphs. \emph{Random Structures \& Algorithms} \tbf{12}, no. 4 (1998), 381--403.

\bibitem{Kwan} M.\ Kwan. Almost all Steiner triple systems have perfect matchings. \emph{arXiv preprint} arXiv:1611.02246 (2016).

\bibitem{LL} N. Linial and Z. Luria. Discrepancy of high-dimensional permutations, \emph{Discrete Analysis} 2016:11, 8pp.

\bibitem{PR} K.\ T.\ Phelps and V.\ R\"odl. Steiner triple systems with minimum independence number. \emph{Ars Combin.} \tbf{21} (1986), 167--172.

\bibitem{W} L. Warnke. On the method of typical bounded differences. \emph{Combinatorics, Probability and Computing} \tbf{25}, no. 2 (2016), 269--299.

\end{thebibliography}
\end{document}